\documentclass{amsart}
\usepackage{amssymb}
\usepackage[all]{xy}

\newtheorem{thm}{Theorem}

\newtheorem{prop}[thm]{Proposition}

\theoremstyle{definition}

\theoremstyle{remark}

\numberwithin{equation}{thm}

\newcommand{\F}{\mathbb{F}}
\newcommand{\N}{\mathbb{N}}

\newcommand{\Q}{\mathbb{Q}}
\newcommand{\Z}{\mathbb{Z}}

\newcommand{\cN}{\mathcal{N}}
\newcommand{\cO}{\mathcal{O}}

\newcommand{\cQ}{\mathcal{Q}}

\newcommand{\gl}{\mathfrak{gl}}

\newcommand{\Qlb}{\overline{\Q}_\ell}

\DeclareMathOperator{\tr}{tr}

\title[Orbit closures in the enhanced nilpotent cone]
{Corrigendum to `Orbit closures in the enhanced nilpotent cone', published in Adv.~Math. 219 (2008)}
\author{Pramod N. Achar}
\address{Department of Mathematics\\
  Louisiana State University\\
  Baton Rouge, LA 70803\\
  U.S.A.}
\email{pramod@math.lsu.edu}
\author{Anthony Henderson}
\address{School of Mathematics and Statistics\\
  University of Sydney, NSW 2006\\
  Australia}
\email{anthony.henderson@sydney.edu.au}


\begin{document}

\begin{abstract}
In this note, we point out an error in the proof of Theorem 4.7 of [P.~Achar and A.~Henderson, \emph{Orbit closures in the enhanced nilpotent cone}, Adv. Math. {\bf 219} (2008), 27--62], a statement about the existence of affine pavings for fibres of a certain resolution of singularities of an enhanced nilpotent orbit closure.  We also give independent proofs of later results that depend on that statement, so all other results of that paper remain valid.
\end{abstract}

\maketitle

The paper \cite{ah} carries out the determination of orbit closures in the enhanced nilpotent cone $V\times\cN$ and their local intersection cohomology. A key role is played by a certain resolution of singularities of the orbit closure $\overline{\cO_{\mu;\nu}}$, denoted $\pi_{\mu;\nu}$.

In \cite[Theorem 4.7]{ah}, we asserted that each fibre $\pi_{\mu;\nu}^{-1}(v,x)$ of this resolution has an affine paving. Regrettably, our proof of this statement was wrong: the error comes four lines after equation (4.9), where we assumed without justification that $x(V_k)\not\supseteq U_{d_k}$. We have not found either a correct proof or a counterexample, so the existence of an affine paving for $\pi_{\mu;\nu}^{-1}(v,x)$ is an open problem in general.

When $v=0$, the fibre $\pi_{\mu;\nu}^{-1}(v,x)$ is a generalized Springer fibre of type A, and an affine paving can be constructed by induction on the length of the partial flag, as shown by Spaltenstein~\cite{spaltenstein}. A similar method works whenever $v\in\ker(x)$, but not for general $(v,x)$, as the example given before \cite[Theorem 4.7]{ah} shows. The methods of \cite{dlp} also appear insufficient. 

In the remainder of \cite{ah}, the affine pavings were used only in the proof of Corollary 4.8(1) and equation (5.7). We will now give an independent proof of these consequences, so that all results stated in \cite{ah} other than \cite[Theorem 4.7]{ah} remain valid.

We use the same notation as in \cite{ah}. The result we must prove (an amalgamation of \cite[Corollary 4.8(1)]{ah} and \cite[(5.7)]{ah}) is as follows.

\begin{thm} \label{thm:substitute}
Let $(\rho;\sigma),(\mu;\nu)\in\cQ_n$. There is a polynomial $\Pi_{\mu;\nu}^{\rho;\sigma}(t)\in\N[t]$, independent of $\F$, satisfying the following two properties.
\begin{enumerate}
\item For any $(v,x)\in\cO_{\rho;\sigma}$,
\[
\begin{split}
&\sum_i \dim H^{2i}(\pi_{\mu;\nu}^{-1}(v,x),\Qlb)\,t^i=\Pi^{\rho;\sigma}_{\mu;\nu}(t),
\\
&\text{ and }H^i(\pi_{\mu;\nu}^{-1}(v,x),\Qlb)=0\text{ for $i$ odd.}
\end{split}
\]
\item If $\F$ is the algebraic closure of $\F_q$, then for any $(v,x)\in\cO_{\rho;\sigma}(\F_q)$,
\[
|\pi_{\mu;\nu}^{-1}(v,x)(\F_q)|=\Pi_{\mu;\nu}^{\rho;\sigma}(q).
\]
\end{enumerate}
\end{thm}

Even though it relies superficially on the equation (5.7) which we are trying to prove, the result \cite[Proposition 5.6]{ah} is still available to use: we need only replace each occurrence of an expression of the form $\Pi_{\mu;\nu}^{\rho;\sigma}(q)$ with the expression which was actually used in the proof, namely $|\pi_{\mu;\nu}^{-1}(v,x)(\F_q)|$ for $(v,x)\in\cO_{\rho;\sigma}(\F_q)$.

We first prove a weaker form of Theorem~\ref{thm:substitute}. Note the change from $\N[t]$ to $\Z[t]$.

\begin{prop} \label{prop:polynomiality}
Let $(\rho;\sigma),(\mu;\nu)\in\cQ_n$. Suppose $\F$ is the algebraic closure of $\F_q$. There is a polynomial $\Pi_{\mu;\nu}^{\rho;\sigma}(t)\in\Z[t]$, independent of $\F$, which satisfies property \textup{(2)} of Theorem \ref{thm:substitute}.
\end{prop}
\begin{proof}
This could be proved by induction on the length of the partial flag as in \cite{spaltenstein}, but it is quicker for us to imitate the proof of \cite[Proposition 5.7]{ah}. Recall that $R$ denotes the ring of all functions $g: \Z_{>0} \to \Qlb$ of the form
\begin{equation*}
g(s) = \sum_i c_i (a_i)^s
\qquad\text{with $c_i \in \Z$ and $a_i \in \Qlb$ (a finite sum),}
\end{equation*}
and $K$ denotes its fraction field. We
identify $\Z[t]$ with a subring of $R$ via the map which sends
a polynomial $p(t)$ to the function $s\mapsto p(q^{s/2})$.

We can define an element $\pi_{\mu;\nu}^{\rho;\sigma}\in R$ by the rule
\[
\pi_{\mu;\nu}^{\rho;\sigma}(s)=|\pi_{\mu;\nu}^{-1}(v,x)(\F_{q^s})|,
\]
where $(v,x)\in\cO_{\rho;\sigma}(\F_q)$. Then \cite[Propositions 5.5 and 5.6]{ah} together imply that equation \cite[(5.5)]{ah} holds in the field $K$, where the left-hand side (which, as stated, involves the polynomials $\Pi_{\mu;\nu}^{\rho;\sigma}(t)$) is replaced by
\[
\sum_{(\tau;\upsilon)\in\cQ_n}\lambda_{(\tau;\upsilon)}\pi_{\mu;\nu}^{\tau;\upsilon}\pi_{\mu';\nu'}^{\tau;\upsilon}.
\]

As in the proof of \cite[Proposition 5.7]{ah}, the uniqueness in \cite[Theorem 5.4]{ah} implies that $\pi_{\mu;\nu}^{\rho;\sigma}$ is an element of $\Q(t)$, and hence of $\Q(t)\cap R=\Z[t,t^{-1}]$. Since $\pi_{\mu;\nu}^{\rho;\sigma}$ is $\Z$-valued, it must lie in $\Z[t]$. Moreover, uniqueness shows that $\pi_{\mu;\nu}^{\rho;\sigma}$ is unchanged under $t\mapsto -t$, so it actually lies in $\Z[t^2]$. Let $\Pi_{\mu;\nu}^{\rho;\sigma}(t)\in\Z[t]$ be the polynomial such that $\Pi_{\mu;\nu}^{\rho;\sigma}(t^2)$ is identified with $\pi_{\mu;\nu}^{\rho;\sigma}$. Then by definition we have
\[
\pi_{\mu;\nu}^{\rho;\sigma}(s)=\Pi_{\mu;\nu}^{\rho;\sigma}(q^s),\text{ for all }s
\in\Z_{>0}.
\]
Uniqueness also implies that $\pi_{\mu;\nu}^{\rho;\sigma}$, and hence $\Pi_{\mu;\nu}^{\rho;\sigma}(t)$, is independent of the prime power $q$ used to define it. This proves the claim. 
\end{proof}

We next prove a purity result, by a standard method. Recall that if $X$ is a projective variety with a Frobenius morphism $F$ relative to the finite field $\F_q$, the cohomology of $X$ is said to be \emph{pure} if the eigenvalues of $F$ on $H^i(X,\Qlb)$ are algebraic numbers all of whose complex conjugates have absolute value $q^{i/2}$.

\begin{prop} \label{prop:purity}
Let $(\rho;\sigma),(\mu;\nu)\in\cQ_n$. Suppose that $\F$ is the algebraic closure of $\F_q$, and let $(v,x)\in\cO_{\rho;\sigma}(\F_q)$. Then the cohomology of $\pi_{\mu;\nu}^{-1}(v,x)$ is pure.
\end{prop}
\begin{proof}
Since $\pi_{\mu;\nu}$ is a resolution of singularities of $\overline{\cO_{\mu;\nu}}$, the derived push-forward complex $R(\pi_{\mu;\nu})_*\Qlb$ is pure of weight $0$ by \cite[Remarque 5.4.9]{bbd}; we must show that it is pointwise pure. This will follow from the general principle \cite[Proposition 2.3.3]{ms}, if we can show the existence of a transverse slice $S$ to the orbit $\cO_{\rho;\sigma}$ at $(v,x)$ and a $1$-parameter subgroup $\varphi:\F^\times\to G\times\F^\times$ which contracts $S$ to $(v,x)$. 

Here we have enlarged the action of $G=GL(V)$ on the enhanced nilpotent cone $V\times\cN$ to $G\times\F^\times$, where $\F^\times$ has the obvious scaling action on $V$ and on $\cN$. It is clear from \cite[Section 2]{ah} that the $(G\times\F^\times)$-orbits in $V\times\cN$ are the same as the $G$-orbits.

By \cite[Proposition 2.3]{ah}, there exists a normal basis $\{v_{ij}\}$ of $V$ for $(v,x)$. Let $\lambda=\rho+\sigma$ be the Jordan type of $x$. Recall from \cite[Proposition 2.8]{ah} that $E^x=\{y\in\gl(V)\,|\,[y,x]=0\}$ has basis
\[
\{y_{i_1,i_2,s}\,|\,1\leq i_1,i_2\leq\ell(\lambda),\,
\max\{0,\lambda_{i_1}-\lambda_{i_2}\}\leq s\leq \lambda_{i_1}-1\},
\]
where
\[
y_{i_1,i_2,s}v_{ij}=\begin{cases}
v_{i_2,j-s},&\text{ if $i=i_1$, $s+1\leq j\leq\lambda_i$,}\\
0,&\text{ otherwise.}
\end{cases}
\]
Let $U$ be the subspace of $\gl(V)$ with basis
\[
\{z_{i_1,i_2,s}\,|\,1\leq i_1,i_2\leq\ell(\lambda),\,
\max\{0,\lambda_{i_1}-\lambda_{i_2}\}\leq s\leq \lambda_{i_1}-1\},
\]
where
\[
z_{i_1,i_2,s}v_{ij}=\begin{cases}
v_{i_1,s+1},&\text{ if $i=i_2$, $j=1$,}\\
0,&\text{ otherwise.}
\end{cases}
\]
We clearly have
\[
\tr(y_{i_1,i_2,s}\,z_{i_1',i_2',s'})=\begin{cases}
1,&\text{ if $i_1'=i_1$, $i_2'=i_2$, $s'=s$,}\\
0,&\text{ otherwise.}
\end{cases}
\]
So the trace form restricts to a perfect pairing $E^x\times U\to\F$. In other words, the subspace $U$ is complementary to $[\gl(V),x]$, which is the subspace perpendicular to $E^x$ for the trace form.

Let $T$ be the subspace of $V$ spanned by $\{v_{ij}\,|\,1\leq i \leq\ell(\sigma), \rho_i+1\leq j\leq\rho_i+\sigma_i\}$. By \cite[Proposition 2.8(5)]{ah}, $T$ is complementary to $E^x v$. It follows immediately that $T\oplus U$ is complementary to $\{(yv,[y,x])\,|\,y\in\gl(V)\}$ in $V\oplus\gl(V)$. Hence 
\[ S=(v+T)\times((x+U)\cap\cN) \] 
is a transverse slice in $V\times\cN$ to the orbit $\cO_{\rho;\sigma}$ at $(v,x)$.

Let $\varphi':\F^\times\to G$ be the $1$-parameter subgroup defined by the rule
\[
\varphi'(t)v_{ij}=t^{j-\rho_i-1}v_{ij}.
\]
Define $\varphi:\F^\times\to G\times\F^\times:t\mapsto(\varphi'(t),t)$. Then by definition, $\varphi(\F^\times)$ fixes $v=\sum v_{i,\rho_i}$ and acts with strictly positive weights on $T$. From the fact that $xv_{ij}$ equals either $v_{i,j-1}$ or $0$, it follows that $\varphi(\F^\times)$ fixes $x$. Finally, we have
\[
\varphi(t)z_{i_1,i_2,s}=t^{\rho_{i_2}-\rho_{i_1}+s+1}z_{i_1,i_2,s},
\]
where the exponent is positive by the assumptions on $s$. So $\varphi(\F^\times)$ acts with strictly positive weights on $U$ also. Hence it contracts $S$ to $(v,x)$ as required.
\end{proof}

We can now give the proof of Theorem \ref{thm:substitute}.

\begin{proof}
Suppose first that $\F$ is the algebraic closure of $\F_q$. Define the polynomial $\Pi_{\mu;\nu}^{\rho;\sigma}$ as in Proposition \ref{prop:polynomiality}. By the Grothendieck Trace Formula,
\[
\Pi_{\mu;\nu}^{\rho;\sigma}(q^s)=\sum_i(-1)^i\tr(F^s\,|\,H^i(\pi_{\mu;\nu}^{-1}(v,x),\Qlb))\text{ for all }s\geq 1.
\]
Proposition \ref{prop:purity} ensures that no Frobenius eigenvalue can occur in more than one cohomology group. We can conclude that every Frobenius eigenvalue arising in the right-hand side is an integer power of $q$. By Proposition \ref{prop:purity} again, every eigenvalue of $F$ on $H^i(\pi_{\mu;\nu}^{-1}(v,x),\Qlb)$ must equal $q^{i/2}$, with $H^i(\pi_{\mu;\nu}^{-1}(v,x),\Qlb)$ vanishing if $i$ is odd. This proves that the polynomial $\Pi_{\mu;\nu}^{\rho;\sigma}$ satisfies property (1) of Theorem \ref{thm:substitute}, and hence has nonnegative coefficients.

Finally, the fact that property (1) holds when $\F$ is an algebraic closure of any finite field implies that it must hold in general, by the well-known principles of \cite[Section 6.1]{bbd}.
\end{proof} 
%%%%%%%%%%%%%%%%%%%%%%%%%%%%%%%%%%%%%%%%%%%%%%%%%%%%%%%%%%%%%%%%%%%%%%%%%%%%

%%%%%%%%%%%%%%%%%%%%%%%%%%%%%%%%%%%%%%%%%%%%%%%%%%%%%%%%%%%%%%%%%%%%%%%%%%%%%

\begin{thebibliography}{BBD}

\bibitem[AH]{ah}
P.~N.~Achar and A.~Henderson, {\it Orbit closures in the enhanced nilpotent cone},
Adv. Math. {\bf 219} (2008), no.~1, 27--62.

\bibitem[BBD]{bbd}
A. Beilinson, J. Bernstein, and P. Deligne, {\it Faisceaux pervers},
Analyse et topologie sur les espaces singuliers, I (Luminy, 1981),
Ast\'erisque 100, Soc. Math. France, Paris, 1982, 5--171.

\bibitem[DLP]{dlp}
C. De~Concini, G. Lusztig, and C. Procesi, {\it Homology of the zero-set of a
nilpotent vector field on a flag manifold}, J. Amer. Math. Soc. {\bf 1} (1988), no.~1,
15--34.

\bibitem[MS]{ms}
J.~G.~M.~Mars and T.~A.~Springer, {\it Hecke algebra representations related to spherical varieties}, Represent. theory {\bf 2} (1998), 33--69.

\bibitem[Spa1]{spaltenstein}
N.~Spaltenstein, {\it The fixed point set of a unipotent transformation on the 
flag manifold}, Nederl. Akad. Wetensch. Proc. Ser. A {\bf 79} = Indag. Math.  
{\bf 38}  (1976), no. 5, 452--456.

\end{thebibliography}
\end{document}